\documentclass[12pt]{amsart} 
\usepackage{amsmath}
\usepackage{amssymb}
\usepackage{amsthm}  
\DeclareMathOperator{\sgn}{sgn}
\usepackage[names,dvipsnames,named]{xcolor} 
\usepackage{lastpage}
\usepackage{caption} 
\usepackage{bigints}
\usepackage[utf8]{inputenc}
\usepackage{textcomp}    				
\usepackage[T1]{fontenc} 				
\usepackage{marvosym}    				
\usepackage[sc]{mathpazo}  	 			
 \usepackage[english]{babel} 
 \usepackage{float} 
 \usepackage{graphics} 
 \usepackage{graphicx} 
\usepackage[utf8]{inputenc}
\usepackage{mathtools}
\usepackage{authblk} 					
\usepackage{lastpage} 					
\usepackage{enumerate}	 				
\usepackage{hyperref} 					
\usepackage{cleveref}
\usepackage{capt-of} 
\usepackage{calrsfs}
\usepackage{lipsum}  
\usepackage{epsfig}
\newcommand{\C}{\mathbf{C}}

\newtheorem{theorem}{Theorem}
\newtheorem{corollary} {Corollary}
\newtheorem{lemma}{Lemma}

\theoremstyle{definition}

\theoremstyle{definition}
\newtheorem{remark}{Remark}
\theoremstyle{definition}
\newtheorem{example}{Example}
\numberwithin{equation}{section}
\numberwithin{table}{section}
\numberwithin{figure}{section}

\newcommand{\CC}{\mathbf{C}}

\DeclareMathAlphabet{\pazocal}{OMS}{zplm}{m}{n}

\numberwithin{equation}{section}

\title{Complex oscillations of non-definite Sturm-Liouville problems, II }
\newcommand{\doi}[1]{\url{https://doi.org/#1}}

\author{Angelo B. Mingarelli}
\address{School of Mathematics and Statistics, Carleton University, Ottawa, Canada, K1S 4P4}
\email{angelo@math.carleton.ca}
\urladdr{https://people.math.carleton.ca/~angelo/} 

\makeatletter
\renewcommand{\maketitle}{\bgroup\setlength{\parindent}{0pt}

\vspace{1truecm}
\begin{center}{\vbox{\titlefont\@title}}\end{center}
\vspace{0.5truecm}
\begin{center}{\@author} \end{center}

\egroup
}

\renewcommand{\@fnsymbol}[1]{%
    \ifcase#1 \or {\,\Letter\!} \or\textasteriskcentered\or \textasteriskcentered\textasteriskcentered 
    \else\@ctrerr\fi}
\makeatother

\newcommand*{\titlefont}{\fontsize{18}{21.6}\selectfont\textbf}

\begin{document}
\maketitle
\pagestyle{plain}

\begin{center}
\noindent
\begin{minipage}{0.85\textwidth}\parindent=15.5pt


{\small{
\noindent {\bf Abstract.} 
We correct and update a result of R.G.D. Richardson \cite{Rich18} dealing with the separation of zeros of the real and imaginary parts of non-real eigenfunctions of non-definite Sturm-Liouville eigenvalue problems. We then extend it to the case where the weight function is allowed to be identically zero on a subinterval that excludes the end-points and study the behavior of the zeros of the real and imaginary parts when the end-points are included. Examples are given illustrating the sharpness of the results along with open questions.

\smallskip

\noindent {\bf{Keywords:}} Sturm-Liouville, eigenvalues, non-definite, complex ghosts, oscillations, zeros, separation theorem
\small

\noindent{\bf{2020 Mathematics Subject Classification:}}  34B24, 34B09, 34C10}}

\end{minipage}
\end{center}


\section{Introduction}


We consider the Sturm-Liouville eigenvalue problem with real piecewise continuous coefficients over an interval $[a, b]$, 
\begin{gather}
-y^{\prime\prime}+q(x)y=\lambda r(x)y, \label{e1} \\
 y(a)=y(b)=0.\label{e1p}
\end{gather}
in the non-definite case, \cite{abm86}, i.e., this is essentially the case where the quadratic forms induced by the two operators on the left and on the right are sign-indefinite on their domains.  The more general equation  $$-(p(x) y^\prime)^\prime +q(x)y=\lambda r(x)y$$ where $p(x)$ is of one sign may be reduced to \eqref{e1} by means of a Liouville transformation, and so leads to no new results. Thus, we shall restrict ourselves to the formulation \eqref{e1} in this paper. 

{\bf Tacit assumptions and Notation}

In the sequel we will always assume that the eigenvalue problem \eqref{e1}-\eqref{e1p} is non-definite and has at least one pair of non-real eigenvalues, usually denoted by $\lambda, \bar{\lambda} \in \C$. The problem of the actual {\it existence} of non-real eigenvalues for problems of this type is difficult and very few results exist in this direction, see, e.g., \cite{alm}, \cite{fs}. Given a non-real eigenvalue of \eqref{e1}-\eqref{e1p}, an associated non-real eigenfunction will be denoted by $y=\varphi + i\psi$, where $\varphi, \psi$ are its, normally unspecified, real and imaginary parts. A non-real eigenfunction will be called a complex ghost state or a {\it complex ghost}, for short \cite{abm86}. Hereinafter, all solutions will be assumed non-trivial unless otherwise specified. In addition, since ${\rm Im}\ \lambda \neq 0$ we will always assume that, unless otherwise specified, ${\rm Im}\ \lambda > 0$.  This is always possible as we can replace $\lambda$ by $\bar{\lambda}$ and $y$ by $\bar{y}$ in \eqref{e1}. 

The following original result is essentially stated in Richardson \cite{Rich18}, see also \cite{abm82,abm86}.
\begin{theorem} {\rm [ \cite{Rich18}, Theorem X]} \label{lem01}
Let $r$ be continuous in $[a, b]$. If $r(x)$ changes its sign precisely once in $(a,b)$ then the roots of the real and imaginary parts of any complex ghost separate one another.
\end{theorem}

A discussion of  Theorem~\ref{lem01} is necessary. First, it was understood at the time of Richardson's paper (as it is now in classical Sturmian cases, such as in Sturm's separation theorem) that the expression {\it separate one another} meant that between any two consecutive zeros of one solution there is always exactly one zero (of any other independent solution) that is distinct from either of the two original zeros. It is unlikely that Richardson had in mind the possibility that the weight function, $r(x)$, could be identically zero in a subinterval, however the latter possibility does change the result(s).

We will show below that Theorem~\ref{lem01} as stated is, in general, not true (see e.g., Example~\ref{exa1} and Figure~\ref{fig31} below). This was alluded to in \cite{abm82} but without discussion. It is true if one excludes the endpoints as one of the two zeros in question, i.e., we are dealing only with interior zeros. However, counterexamples can be found in cases where the endpoints are included (see the examples at the end of this article). Indeed, we show in Remark~\ref{remtwo} that if $r(x)$ is identically zero in a right-neighborhood of $x=a$ that includes the endpoint, then Theorem~\ref{lem01} is false. In fact, we go further and show in Lemma~\ref{lex} below, that the statement in Theorem~\ref{lem01} cannot be true if the endpoints are included (if even the weight function never vanishes identically in any subinterval). As a result, Richardson's Theorem X in \cite{Rich18} must be replaced by something like Theorem~\ref{th3}. Indeed, in Theorem~\ref{th3} we show that the single turning point or zero required in Theorem~\ref{lem01} may be replaced by an {\it interval of zeros} without affecting the conclusion.

In addition, we will find conditions under which given two consecutive zeros (one of which is an endpoint) of either the real (or imaginary) part of a complex ghost there is or there isn't a zero of the imaginary (or real) part. In this sense, this paper continues a study of the zeros of the real and imaginary parts of such ghosts, research that began with \cite{Rich18}, \cite{vs}, \cite{abm82}, \cite{maps} and continued with \cite{Mervis2016}, \cite{MK1} and the results therein.

\begin{remark} We note that Theorem~\ref{lem01} appears to be exceptional in the sense that it has no apparent equivalent in the two turning point case (with a weight function that doesn't vanish identically in a subinterval). For example, in \cite{Mervis2016} there are examples showing that the interior zeros of the real and imaginary parts of non-real eigenfunctions may not separate one another.
\end{remark}

\section{Main results}

First, we show that Theorem~\ref{lem01} cannot be true if one or more end-points are included in its conclusion.

\begin{lemma}\label{lex}
There is no non-trivial complex ghost whose real and imaginary parts have zeros that separate one another on an interval that includes the endpoints of $[a, b]$.
\end{lemma}

\begin{proof}
Without loss of generality we may assume that the endpoint in question is $x=a$. Assuming the contrary, let $y= \varphi + \psi i$ be a non-trivial complex ghost where $\varphi$ satisfies $\varphi(a)=0$ and $\varphi(x_0)=0$ where $x_0$ is a zero of $\varphi$ larger than $x=a$. Then, by hypothesis, $\psi(x)$ has a unique zero $t_0$ in $(a, x_0)$. Now $\psi(a)=\psi(t_0)=0$, thus, by assumption $\varphi(x_1)=0$ for some unique $x_1 \in (a, t_0)$. Since $\varphi(a)=\varphi(x_1)=0$ there must be a unique zero $t_1 \in (a, x_1)$ of $\psi$. Continuing in this way we see that the now infinite sequence of zeros $x_i, t_i$ must accumulate at $x=a$. On the other hand the theorem of the mean then implies that both $\varphi^\prime(a)=0$ and $\psi^\prime(a)=0$. Thus, $y^\prime(a)=0$ and so $y$ is the trivial solution, a contradiction.
\end{proof}

\begin{lemma}\label{lem0}
Let $\lambda$ be a non-real eigenvalue of \eqref{e1}-\eqref{e1p} and $y = \varphi +i \psi$ be a corresponding eigenfunction. Then $\varphi, \psi$ are linearly independent over $[a, b]$.
\end{lemma}
\begin{proof}
Assume the contrary. Then $\psi =\mu \varphi$ where $\mu \in \CC \textbackslash {0}$.  Since $y=(1+ i\mu)\varphi$ satisfies \eqref{e1} we get that $- \varphi^{\prime\prime} + q\varphi = \lambda r \varphi$. Since the left side of the previous equation is real and ${\rm Im}\, \lambda \neq 0$ we find that $r(x)\varphi(x) \equiv 0$ on $[a, b]$. Since, by hypothesis, \eqref{e1}-\eqref{e1p} is non-definite, $r(x)$ cannot be identically zero and so there is an interval $J\subset [a, b]$ on which $r(x) > 0$, say. Thus, $\varphi(x) \equiv 0$ on $J$. Hence $\psi(x) \equiv 0$ on $J$ so that $y \equiv 0$ on $J$. This, however, violates the uniqueness of solutions of an initial value problem with initial conditions on a point of $J$. This contradiction proves the lemma.
\end{proof}

\begin{lemma}\label{lem3a} Let $c>a$, $r$ be piecewise continuous  on $(a, c)$ and $r(x)< 0$ there.  In addition, let $r(x) \equiv 0$ on $[c, d] \subset (a, b)$ and $r(x) > 0$ and continuous on $(d, b)$.  Then, for any complex ghost $y$ of \eqref{e1}, 
\begin{equation}\label{gfn}
G(x) \equiv \int_{a}^{x} r |y|^2\ dt,\quad\quad x\in [a, b],
\end{equation}
is of one sign in $(a, b)$, in fact, $G(x)<0$, for all $x\in (a, b)$.
\end{lemma}

\begin{proof} Since $r(x)<0$ on $(a, c)\subseteq (a, d)$, by hypothesis, then $G(x) < 0$ too, as an eigenfunction cannot vanish identically on an interval unless it is identically zero there. Assume, on the contrary, that there is a smallest point $x_1\in (a, b)$ with $G(x_1) =0$. Since $G(x)$ must change its sign around $x_1$ it follows that $x_1 \in [d, b)$. 

Next, observe that $G(x_1)=G(b)=0$ and $G \in C^1(d,b)$.  Hence there exists $\eta \in (d, b)$ such that $r(\eta)|y(\eta)|^2 =0$, i.e., $y(\eta)=0$. Consider the boundary problem \eqref{e1} subject to the boundary conditions $y(\eta)=y(b)=0$. This new eigenvalue problem admits our complex ghost $y$ as an eigenfunction corresponding to the same non-real eigenvalue. Thus, this problem must be non-definite on $[\eta, b]$, i.e., $r(x)$ must change its sign on this interval, but this is a contradiction as $r(x)>0$ there. Hence no such $x_1$ can exist. This completes the proof.
\end{proof}

The next result is an immediate consequence of Lemma~\ref{lem3a}.

\begin{corollary}\label{coro1}
If $r(x) > 0$ on $(a, c)$, $r(x) \equiv 0$ on $[c, d] \subset (a, b)$, and $r(x) < 0$ on $(d, b)$, then $G(x)>0$ on $(a, b)$.
\end{corollary}

\begin{proof} Replace $\lambda$ by $-\lambda$ and $r$ by $-r$ in Lemma~\ref{lem3a} so that the eigenfunction doesn't change.
\end{proof}

\begin{corollary} (Theorem X, \cite{Rich18})  If $r(x)$ is not identically zero in any right interval of $a$ and $r(x)$ has a unique zero around which $r(x)$ changes sign in $(a, b)$, then $G(x)$ is of one sign in $(a, b)$.
\end{corollary}

\begin{proof} This is the case where $c=d$ in Lemma~\ref{lem3a} so either Lemma~\ref{lem3a} or Corollary~\ref{coro1} applies.
\end{proof}

\begin{corollary}\label{ht}
Let $c>a$, $r$ be piecewise continuous on $(a, c)$ and $r(x)<0$ there.  In addition, let $r(x) \equiv 0$ on $[c, d] \subset (a, b)$ and $r(x) > 0$ and continuous on $(d, b)$.  Then, for any $x\in (a, b)$, the function $H(x) = \int_x^b r|y|^2\, dt$ is of one sign on $(a, b)$, i.e., $H(x) > 0$ on $(a, b)$.
\end{corollary}

\begin{proof} This is an application of Lemma~\ref{lem3a} and the fact that $\int_a^b r|y|^2\, dx = 0$ must hold for any complex ghost.
\end{proof}
  
 \begin{lemma}\label{lemi}
 Let $y = \varphi+ i \psi$ be a complex ghost. Then, for any two points $x_1, x_2 \in [a, b]$ we have,
 \begin{equation}
\label{eid}
\left [\varphi^\prime\psi - \varphi\psi^\prime\right] \bigg |_{x_1}^{x_2} = {\rm Im}\, \lambda \int_{x_1}^{x_2} r|y|^2\, \ dx.
\end{equation}
 \end{lemma}
 
 \begin{proof}
 Let $y = \varphi+ i \psi$ be a complex ghost. Then, from \eqref{e1}, we find that (suppressing the variables for simplicity of exposition),
\begin{gather}
-\varphi^{\prime\prime} +(q - r\, {\rm Re}\ \lambda) \varphi = - ({\rm Im}\, \lambda)\ r\ \psi,\label{phiq} \\
-\psi^{\prime\prime} +(q - r\,{\rm Re}\ \lambda)\psi =  ({\rm Im}\, \lambda)\ r\ \varphi,\label{psiq}
\end{gather}
where ${\rm Im}\, \lambda \neq 0$ by hypothesis. For any two points $x_1, x_2 \in [a, b]$ a simple Sturmian argument applied to \eqref{phiq} and \eqref{psiq} now yields,
\begin{equation}
\label{eid}
\left [\varphi^\prime\psi - \varphi\psi^\prime\right] \bigg |_{x_1}^{x_2} = {\rm Im}\, \lambda \int_{x_1}^{x_2} r|y|^2\, \ dx.
\end{equation}
From \eqref{eid} since $y(a)=0$ is equivalent to $\varphi(a)=\psi(a)=0$ we get that, for any $x\in [a, b]$,
\begin{equation}
\label{eid2}
\varphi^\prime(x)\psi(x) - \varphi(x)\psi^\prime(x) = {\rm Im}\, \lambda \int_{a}^{x} r |y|^2\, \ dt = ({\rm Im}\, \lambda)\ G(x).
\end{equation}
By Lemma~\ref{lem3a} and Corollary~\ref{coro1} $G(x)$ is of one sign only on $(a, b)$ (depending on the sign of $r(x)$ on $(a, c)$).
 \end{proof}
 
 The next lemma is an application of classical folklore so a proof is given for the sake of completeness.
 
 \begin{lemma}\label{lemii}
 Let $q, r \in C[a,b]$, and for some $c>a$,  let $r(x)<0$ on $(a, c)$, $r(x) \equiv 0$ on $[c, d] \subset (a, b)$ and $r(x) > 0$ and continuous on $(d, b)$.  Let $\varphi, \psi$ be two non-trivial solutions of the system \eqref{phiq} and \eqref{psiq}. If the zeros of $\varphi$ (or $\psi$) have a point of accumulation $x_0 \in (a, b)$, then the non-real eigenfunction, $y$, satisfies  $y(x_0)=0$. In addition, the end-points, $x=a, b$, cannot be points of accumulation of either $\varphi$ or $\psi$.
 \end{lemma}
 
 \begin{proof} By fundamental existence theorems we know that the continuity of the coefficients $q, r$ implies that both solutions $\varphi, \psi \in C^2(a,b)$ (whenever they exist, or whenever there are non-real eigenvalues).
 
If possible let $x_0\in (a, b)$ be a point of accumulation of $\varphi$ and let $x_n\to x_0$ be an infinite sequence of zeros of $\varphi$ accumulating at $x_0$. Then, by continuity, $\varphi(x_0)=0$. On the other hand, $\varphi\in C^1$ so that $0 = (\varphi(x_n)-\varphi(x_0))/(x_n-x_0) = \varphi^\prime(\xi_n)$ holds for all $n$ for some $\xi_n$ between $x_0$ and $x_n$. There follows that $\xi_n\to x_0$ as well and thus, since $\varphi\in C^1$, $\varphi^\prime(x_0)=0$. Since $\varphi\in C^2$ we get that $0 = (\varphi^\prime(\xi_n)-\varphi^\prime(x_0))/(\xi_n-x_0) = \varphi^{\prime\prime}(\zeta_n)$ for some $\zeta_n$ between $\xi_n$ and $x_0$. Since $\zeta_n \to x_0$ there follows that $\varphi^{\prime\prime}(x_0)=0.$ From this and \eqref{phiq} we obtain, $r(x_0)\psi(x_0)=0$ (as ${\rm Im}\, \lambda > 0$, by hypothesis.)
 
So, $\psi(x_0)=0$ whenever $r(x_0)\neq 0$. On the other hand, if $r(x)\equiv 0$ on some subinterval $[c,d]\subset (a, b)$ where, in addition, $x_0 \in [c, d]$, then, using \eqref{phiq} and \eqref{psiq}, we get that both $\varphi, \psi$ are solutions of the same equation, namely, $-z^{\prime\prime} +(q - r\,{\rm Re}\ \lambda)z =  0$, where $x\in [c, d]$. Thus either $\varphi, \psi$ are linearly dependent, in which case $\varphi(x_0)=0 =\psi(x_0)$  is immediate, or they are linearly independent, in which case Sturm's Separation Theorem implies the interlacing of the zeros of $\varphi, \psi$ and thus, once again we must have $\psi(x_0)=0.$ Therefore, in any case we must have $\psi(x_0)=0$. Finally, $\varphi(x_0)=\psi(x_0)=0$ is equivalent to $y(x_0)=0$ as was to be shown.

The same argument applies to $\psi$ so that all its interior zeros are isolated as well.

If either $x_0=a$ or $x_0=b$ is an accumulation point of $\varphi$, we proceed as follows: Assume that $x=a$ is such a point. Since $y^\prime(a)\neq 0$ (as $y$ is non-trivial) at least one, maybe both, of $\varphi^\prime(a), \psi^\prime(a)$ are non-zero. We may assume that $\varphi^\prime(a)\neq 0$ (or else if  $\varphi^\prime(a)=0$, then $\psi^\prime(a)\neq 0$ and we can apply the following argument to $\psi$ as $\psi(a)=0$). Applying the argument in the second paragraph of the proof of Lemma~\ref{lemii} above, we can conclude that $\varphi^\prime(a)=0$, which is a contradiction.  Hence $x=a$ cannot be a point of accumulation of $\varphi$ (similarly for $\psi$). If $x=b$ a similar argument applies whose proof is omitted.
 \end{proof}
 
 \begin{corollary}
Let $q, r$ be continuous on $[a, b]$. For some $c>a$, let $r(x) < 0$ on $[a, c)$, $r(x) \equiv 0$ on $[c, d] \subset (a, b)$, and $r(x) > 0$ on $(d, b]$.  Then the interior zeros of any solution $\varphi$ of \eqref{phiq}-\eqref{psiq} are isolated.
 \end{corollary}
 
 \begin{proof}
Since we are dealing with interior zeros, by Lemma~\ref{lemii} if there is such an accumulation point $x_0$ of $\varphi$ that satisfies $a< x_0 < b$, then it must also satisfy $y(x_0)=0$ where $y=\varphi + i \psi$ is an eigenfunction of \eqref{e1}-\eqref{e1p}. Thus, $y(a)=y(x_0)=y(b)=0$. Note that if $x_0 \in (a, d]$  then $r(x)\leq 0$ in which case the boundary problem \eqref{e1} subject to the boundary condition $y(a)=y(x_0)=0$ must have a non-real eigenvalue which isn't possible as the eigenvalue problem on $(a, x_0]$ is now right-definite. On the other hand, if $x_0\in [d, b)$ then $r(x)> 0$ in $[d, b)$ and so in $[x_0, b)$. Hence, the problem \eqref{e1} subject to the boundary condition $y(x_0)=y(b)=0$ must have also have a non-real eigenvalue which isn't possible as this problem is also right-definite. Thus all interior zeros of $\varphi$ are isolated. \end{proof}

\begin{theorem}\label{th3} Let $q, r \in C[a,b]$, and for some $c>a$,  let $r(x)<0$ on $(a, c)$, $r(x) \equiv 0$ on $[c, d] \subset (a, b)$ and $r(x) > 0$ and continuous on $(d, b)$.  Then the interior zeros of the real and imaginary part of a complex ghost separate one another.
\end{theorem}

\begin{proof} Since, by lemma~\ref{lemii}, the interior zeros of $\varphi$ are isolated, we let $x_1 < x_2\in (a, b)$ be two consecutive zeros of $\varphi(x)$ where, without loss of generality, we may take it that $\varphi(x) > 0$ in $(x_1, x_2)$. (If $\varphi$ is negative, we replace $\varphi$ by $-\varphi$ and $\psi$ by $-\psi$ in what follows.) Assume, if possible, that $\psi(x) \neq 0$ in $[x_1, x_2]$. 

Since $r(x)< 0$ in $(a, c)$, Lemma~\ref{lem3a} implies that $G(x)< 0$ in $(a, b)$. Assume, if possible, that $\psi(x) > 0$ in $[x_1, x_2]$. Using \eqref{eid2} we get, for $x_1 \leq x \leq x_2$,
\begin{equation}\label{ww}
\left (\frac{\varphi}{\psi}\right)^\prime = \frac{{\rm Im}\, \lambda}{\psi^2(x)}\ G(x) < 0.
\end{equation}
Thus, $0= \varphi(x_1)/\psi(x_1)  > \varphi(x)/\psi(x)$ which is impossible. Hence, $\psi(x) > 0$ in $[x_1, x_2]$ cannot hold, i.e., $\psi(x)$ must vanish somewhere in $[x_1, x_2]$. A similar argument holds in the case where $\psi(x) <0$ in $[x_1, x_2]$ as from \eqref{ww}, $ \varphi(x)/\psi(x)  > \varphi(x_2)/\psi(x_2) = 0$, which is also impossible.

Now we show that $\psi(x_1)\psi(x_2)\neq 0$. Assuming, if possible, that $\psi(x_1)=0$, then $y(x_1) = 0$ as well. Since $y(a)=y(x_1) = 0$ and $y$ is a non-real eigenfunction, \eqref{e1} must be non-definite on $[a, x_1]$, so $r(x)$ cannot have a fixed sign there. Therefore, we must have that $x_1\in (d, b]$ (see the hypotheses on $r(x)$). Once again the problem \eqref{e1}, subject to the boundary conditions $y(x_1)=0, y(b)=0$, must be non-definite. But this is impossible since $r(x) < 0$ on $(d, b)$.  This contradiction implies that  $\psi(x_1) \neq 0$. A similar argument applies to the case where $\psi(x_2)=0$, and we leave this to the reader.

It follows from the preceding argument that $\psi(x) =0$ somewhere in $(x_1, x_2)$, say, at $x=t_1$. We take it that $t_1$ is the smallest such zero, so that $\psi(x) \neq 0$ in $(t_1, t_2)$. However, since $\varphi(x) \neq 0$ in $[t_1, t_2]$, a modification of \eqref{ww} using \eqref{eid2}, gives
\begin{equation}\label{www}
\left (\frac{\psi}{\varphi}\right)^\prime =  - \frac{{\rm Im}\, \lambda}{\varphi^2(x)}\ G(x) > 0,
\end{equation}
from which we obtain a contradiction as before, as $\psi/\varphi$ must be increasing in $(t_1, t_2)$, yet $\psi$ vanishes at the end-points. This contradiction shows that between any two consecutive interior zeros of $\varphi$, the function $\psi$ has a unique zero.
\end{proof}

\begin{remark}\label{remtwo}
This extends and corrects Theorem~\ref{lem01} from a single turning point (or a zero around which $r(x)$ changes sign) to an interval of zeros. The condition that, for some $c>a$,  $r(x)\neq 0$ on $(a, c)$ is necessary. This is because if $r$ vanishes identically in some subinterval of $[a, c]$ containing $x=a$,  the separation property of the roots fails as $\varphi, \psi$ must be linearly dependent on account of \eqref{eid} applied to said interval.
\end{remark}

\begin{lemma}\label{lem1b}
Let $G(x)<0$ for all $x\in (a, b).$ Let $\varphi^{\prime}(a)>0$, (resp. $\varphi^{\prime}(a)< 0$) and let $x_0 \in (a, b)$ be the smallest zero of $\varphi(x)$.
\begin{enumerate}
\item If $\psi^{\prime}(a)>0$, then $\psi(x)>0$ in $(a, x_0]$  (resp. if $\psi^{\prime}(a)<0$, then $\psi(x) < 0$ in $(a, x_0]$)
\item If $\psi^{\prime}(a)<0$, then $\psi(x)=0$  exactly once in $(a, x_0)$ (resp. if $\psi^{\prime}(a)>0$, then $\psi(x)=0$  exactly once in $(a, x_0)$)
\item If $\psi^{\prime}(a)=0$, then one of the first two cases must occur.
\end{enumerate}
\end{lemma}

\begin{proof} By Lemma~\ref{lemii} we can let $x_0$, where $a< x_0 < b$, be the smallest zero of $\varphi$ (assuming there is a zero at all). Since $\varphi^\prime(a)>0$ and $x_0$ is the smallest zero of $\varphi(x)$, then $\varphi(x)>0$ on $(a, x_0)$. Assume, for the moment that $\psi(x_0)\leq 0$.  Geometrical considerations give us that $\varphi^\prime(x_0)\psi(x_0)\geq 0$. Applying \eqref{eid2} with $x = x_0$ we obtain a contradiction. Thus, it must be the case that $\psi(x_0)> 0$. 

There are now three possibilities, each of which may occur (see Example~\ref{exa1}, Example~\ref{exa2} and Example~\ref{exa3}, below).

Case 1. Let $\psi^\prime(a)>0$. Then $\psi$ increases in a right-neighborhood of $x=a$ and we will show that $\psi(x)\neq 0$ in $(a, x_0)$, so that $\psi(x)> 0$ there.

For assume, on the contrary, that $\psi(x_1)=0$ at some smallest point $x_1$ where $a<x_1<x_0$ and $\psi(x)>0 $ for $a<x<x_1$.  From \eqref{eid2} we know that 
$$\varphi^\prime(x_1)\psi(x_1) - \varphi(x_1)\psi^\prime(x_1) = {\rm Im}\, \lambda \int_{a}^{x_1} r |y|^2\, \ dx<0.$$
which gives a contradiction since $\varphi(x_1)>0$ and $\psi^\prime(x_1) < 0$ (Note that $\psi^\prime(x_1) = 0$ is excluded from the previous display.) Hence, in this case, $\psi(x) \neq 0$ whenever $a< x<x_0$, i.e., $\psi(x)>0$ in $(a, x_0]$.

Case 2. Let $\psi^\prime(a)<0$. In this case, $\psi$ decreases in a right neighborhood of $x=a$ and so $\psi(x)$ must be negative. On the other hand  $\psi(x_0) >  0$ and so $\psi(x_3)=0$ at some smallest point $x=x_3$, where $a<x_3<x_0$ and $\psi(x)<0$ on $(a, x_3)$. Note that $\psi$ cannot vanish again for $x>x_3$ for, if  $\psi(x_4)=0$, where $x_3<x_4<x_0$, then this would violate the separation of internal zeros property in Theorem~\ref{th3} as $\varphi(x)>0$ in $[x_3, x_4]$.

Case 3.  Let $\psi^\prime(a) = 0$. This final case is shown by noting that $\psi(x)$ must be of one sign in a right-neighborhood of $x=a$ after which one of the proofs in cases 1 and 2 must apply, depending on the sign.

The alternative (``respective") result is handled from the first by replacing $y$ by $-y$, or $\varphi, \psi$ by $- \varphi, - \psi$, a transformation which keeps \eqref{eid} invariant, and the eigenvalue the same.
\end{proof}
\begin{corollary}\label{co4}
Let $G(x)<0$ for all $x\in (a, b).$ Let $\psi^{\prime}(a)<0$, and let $x_0 \in (a, b)$ be the smallest zero of $\psi(x)$.
\begin{enumerate}
\item If $\varphi^{\prime}(a)>0$, then $\varphi(x)>0$ in $(a, x_0)$
\item If $\varphi^{\prime}(a)<0$, then $\varphi(x)=0$  exactly once in $(a, x_0)$
\item If $\varphi^{\prime}(a)=0$, then one of the first two cases must occur.
\end{enumerate}
\end{corollary}
\begin{proof} Replace $y=\varphi+i\psi$ by $iy = -\psi+i\varphi$ which is another eigenfunction satisfying \eqref{e1}-\eqref{e1p} and apply the lemma to this eigenfunction. Observe that this change of variable also keeps \eqref{eid} invariant.
\end{proof}

\begin{remark}\label{r2a}
Observe that the proof of Lemma~\ref{lem1b} fails if $x_0=b$ as then $G(b)=0$ so that the right-hand side of \eqref{eid2} is zero. 
\end{remark}

\begin{remark}\label{r2}
Note that if $\psi^\prime(a) = 0$, then one of Case 1, 2 above must apply, i.e., either $\psi$ increases or decreases in a right neighborhood of $x=a$. In either case we have the same results.  See Example~\ref{exa1} for an application of Corollary~\ref{co4}(1)  and  Example~\ref{exa2} for an application of Lemma~\ref{lem1b} (2).
\end{remark}

Results analogous to Lemma~\ref{lem1b} and Corollary~\ref{co4} may be formulated in the event that information is available at the right endpoint. Two samples follow, others may be readily formulated using the change of variables in Lemma~\ref{lem1b}  and Corollary~\ref{co4}. We sketch the proofs.

\begin{lemma}\label{lem5}
Let $G(x)<0$ for all $x\in (a, b).$ Let $\varphi^{\prime}(b) < 0$, (resp. $\varphi^{\prime}(b) > 0$) and let $x_0 \in (a, b)$ be the largest zero of $\varphi(x)$.
\begin{enumerate}
\item If $\psi^{\prime}(b)>0$, then $\psi(x)<0$ in $(x_0, b)$ (resp.  if $\psi^{\prime}(b)<0$ then $\psi(x) > 0$ in $(x_0, b)$ )
\item If $\psi^{\prime}(b)<0$, then $\psi(x)=0$  exactly once in $(x_0, b)$ (resp. if $\psi^{\prime}(b)>0$, then $\psi(x)=0$  exactly once in $(x_0, b)$ ) 
\item If $\psi^{\prime}(b)=0$, then one of the first two cases must occur.
\end{enumerate}
\end{lemma}

\begin{proof} Let $a< x_0 < b$ be the largest zero at which $\varphi$ vanishes (again, assuming there is a zero at all). Since $\varphi^\prime(b)< 0$ and $x_0$ is the largest zero of $\varphi(x)$, then $\varphi(x)>0$ on $(x_0, b)$. Assume, for the moment that $\psi(x_0)> 0$.  Geometrical considerations give us that $\varphi^\prime(x_0)\psi(x_0)\geq 0$. Applying \eqref{eid2} with $x = x_0$ we obtain a contradiction. Thus, it must be the case that $\psi(x_0)\leq 0$. On the other hand, if $\psi(x_0) = 0$, then the left hand side of \eqref{eid2} is zero (at $x=x_0$) while the right hand is non-zero. This contradiction proves that $\psi(x_0) <  0$.

Case 1. Let $\psi^\prime(b)>0$. Then $\psi$ increases in a left-neighborhood of $x=b$ and we will show that $\psi(x)\neq 0$ in $(x_0, b)$, so that $\psi(x) <  0$ there.

For assume, on the contrary, that $\psi(x_1)=0$ at some largest point $x_1$ where $x_0<x_1<b$ and $\psi(x) < 0 $ for $x_1<x<b$.  From \eqref{eid2} we know that 
$$\varphi^\prime(x_1)\psi(x_1) - \varphi(x_1)\psi^\prime(x_1) = {\rm Im}\, \lambda \int_{a}^{x_1} r |y|^2\, \ dx<0.$$
which gives a contradiction since $\varphi(x_1)>0$ and $\psi^\prime(x_1) < 0$ (Note that $\psi^\prime(x_1) = 0$ is excluded from the previous display.) Hence, in this case, $\psi(x) \neq 0$ whenever $x_0< x<b$, i.e., $\psi(x) < 0$, $x \in (x_0, b)$.
\end{proof}

Case 2. If $\psi^\prime(b) < 0$ then $\psi$ decreases and is positive in a left-neighborhood of $x=b$. However, since  $\psi(x_0) <  0$, it follows that $\psi(x_5)=0$ for some $x_0 < x_5< b$, where $x_5$ is taken to be the largest (in fact, the only) such zero, i.e., $\psi(x)>0$ in $(x_5, b)$. If possible, assume that $\psi(x)$ vanishes again at $x=x_6$, where now $x_0< x_6<x_5$. This now violates the separation of zeros property in Theorem~\ref{th3} as these are internal zeros and $\varphi$ is not zero there. 

Case 3. This is shown as in the corresponding cases in the above lemmas, so the proofs are left to the reader.

\begin{corollary}
Let $G(x)< 0$ for all $x\in (a, b).$ Let $\psi^{\prime}(b)> 0$, and let $x_0 \in (a, b)$ be the largest zero of $\psi(x)$.
\begin{enumerate}
\item If $\varphi^{\prime}(b)>0$, then $\varphi(x)<0$ in $(x_0, b)$
\item If $\varphi^{\prime}(b)<0$, then $\varphi(x)=0$  exactly once in $(x_0, b)$
\item If $\varphi^{\prime}(b)=0$, then one of the first two cases must occur.
\end{enumerate}
\end{corollary}

\begin{proof} Consider the eigenfunction $iy = -\psi+i\varphi$ as before.
\end{proof}

We recall that if $G(x)>0$, analogous results may be obtained by replacing $\lambda$ by $-\lambda$ and $r$ by $-r$ in the above theorems and lemmas as the eigenfunction remains the same although the non-real eigenvalue is the negative of the original one corresponding to a different weight function, $r$.

\section{Examples}
In the examples that follow the coefficients $q, r$ appearing in \eqref{e1} are piecewise continuous rather than continuous, this isn't a real lack of generality as one expects the same results holding in this slightly weaker case of piecewise continuity as continuous functions on $[a, b]$ are dense in $L(a, b)$. Thus, solutions of problems with piecewise continuous coefficients may be approximated arbitrarily closely by a differential equation with continuous coefficients. 

In this section we shall make use of the following result,
\begin{lemma}\label{lem4}[\cite{vs}, Theorem 4.2.1] 
Let $p(x) > 0$ a.e. on $[a, b]$ , $q \in L(a, b)$  and $r(x)$ is a real-valued function defined on $[a, b]$ and such that $r(x)$ takes both signs on some subsets of positive measure. Then the eigenvalue problem 
\begin{equation}\label{e11}
-(p(x)y^\prime)^\prime+q(x)y=\lambda r(x)y, \quad y(a)=y(b)=0,
\end{equation}
possesses at most a finite number of non-real eigenvalues. If we let
M the number of pairs of distinct non-real eigenvalues of \eqref{e11},
N the number of distinct negative eigenvalues of \eqref{e11} with $r(x)\equiv 1$,  which we know is finite by classical Sturm-Liouville theory, [\cite{34}, \S 5.8, Theorem 2], then $M \leq  N$.
\end{lemma}

\begin{example}\label{exa1}
As an application of Lemma~\ref{lem3a} and the second conclusion in Lemma~\ref{lem1b}, consider the eigenvalue problem for the equation
$$-y^{\prime\prime}(x) - q\,y(x)= \lambda \sgn(x)y(x), \quad\quad y(-1)=0,\ y(1)=0.$$
where $\lambda \in \C$, $q>0$ is a constant, and $\sgn(x)= -1$ for $x \leq 0$ and $\sgn(x)=  1$ for $x > 0$. Its eigenfunctions are given by

$f\! \left(x\right)=\left\{\begin{array}{cc}
\sin\! \left(\sqrt{-q-\lambda}\, \left(1+x\right)\right) & -1\le x\le 0 
\\
 \sin\! \left(\sqrt{-q-\lambda}\right) \cos\! \left(x \sqrt{-q+\lambda}\right)+
\\
\frac{\sqrt{-q-\lambda}\, \cos\left(\sqrt{-q-\lambda}\right) \sin\left(x \sqrt{-q+\lambda}\right)}{\sqrt{-q+\lambda}} & 0\le x\le 1 
\end{array}\right.$

where the $\lambda\in \C$ satisfy the dispersion relation, see \cite{Jabon84},
\begin{gather}\sqrt{-z -q}\, \cos (\sqrt{-z -q}) \sin (\sqrt{z -q})\\ 
+\sqrt{z -q}\, \cos (\sqrt{z -q}) \sin (\sqrt{-z -q})=0\\
\end{gather}
This problem always has a doubly infinite sequence of real eigenvalues along with a finite number of complex eigenvalues,  \cite{abm86},\cite{Rich18}, one of which is, in the case where $q=40$, 
$\lambda \approx 26.9376 + 6.9215 i$. 

Writing $y = \varphi  + \psi i$, the graphs of $\varphi, \psi,$ and $G(x)$ are reproduced below,

\includegraphics[width=3in]{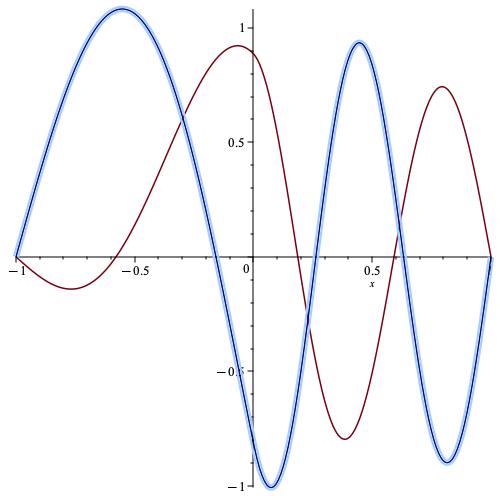}
\includegraphics[width=3in]{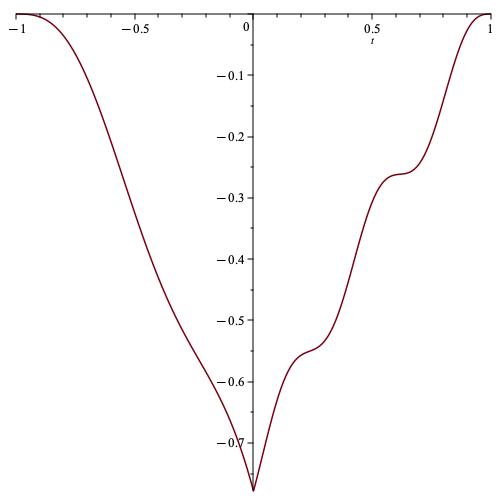}
\captionof{figure}{Here, $\varphi$ is the left-most curve of the graph on the left, $\psi$ being the other. The graph on the right is that of $G(x)$.
 \label{fig31}}
\vskip0.25in

\end{example}

\begin{example}\label{exa2}
consider the initial value problem for the equation
$$-y^{\prime\prime}(x) - (7+\lambda x)y(x)=0, \quad\quad y(-1)=0,\ y(1)=0,$$
where $\lambda \in \C$, is solvable in terms of Airy functions, so that if $y^\prime(-1)=1$, then the solution is
$$\frac{\left(-\mathrm{AiryAi}\! \left(\frac{\lambda -7}{\lambda^{\frac{2}{3}}}\right) \mathrm{AiryBi}\! \left(-\frac{\lambda  x+7}{\lambda^{\frac{2}{3}}}\right)+\mathrm{AiryBi}\! \left(\frac{\lambda -7}{\lambda^{\frac{2}{3}}}\right) \mathrm{AiryAi}\! \left(-\frac{\lambda  x+7}{\lambda^{\frac{2}{3}}}\right)\right) \pi}{\lambda^{\frac{1}{3}}}.$$

The associated Dirichlet problem on $[-1,1]$, i.e.,
$$y(-1)=0=y(1)$$
then gives a doubly infinite sequence of real eigenvalues along with a finite number of complex eigenvalues,  \cite{abm86},\cite{Rich18}, one of which is
$\lambda \approx 12.3076$i. 

Now, let $y(-1)=0,\ y^\prime(-1)=1-i$. Of course, the spectrum is the same but the eigenfunctions differ slightly. In this case, the eigenfunctions take the form,
$$y= -\frac{\mathrm{AiryBi}\! \left(\frac{\lambda -7}{\lambda^{\frac{2}{3}}}\right)\cdot \alpha \cdot \mathrm{AiryAi}\! \left(-\frac{\lambda \cdot x+7}{\lambda^{\frac{2}{3}}}\right)}{\mathrm{AiryAi}\! \left(\frac{\lambda -7}{\lambda^{\frac{2}{3}}}\right)}+\alpha \cdot \mathrm{AiryBi}\! \left(-\frac{\lambda \cdot x+7}{\lambda^{\frac{2}{3}}}\right)$$
where
$$\alpha =\frac{\left(i-1\right)\cdot \mathrm{AiryAi}\left(\frac{\lambda -7}{\lambda^{\frac{2}{3}}}\right) \pi}{\lambda^{\frac{1}{3}}}.$$
The plots of the real and imaginary parts of $y$ are given below in Figure~\ref{fig:airy32} and the qualitative behavior of these functions is in conformity with the conclusions obtained in Lemma~\ref{lem1b} (2) and Lemma~\ref{lem5} (1). 
\begin{figure}[h]
\begin{center}
\includegraphics[width=2in]{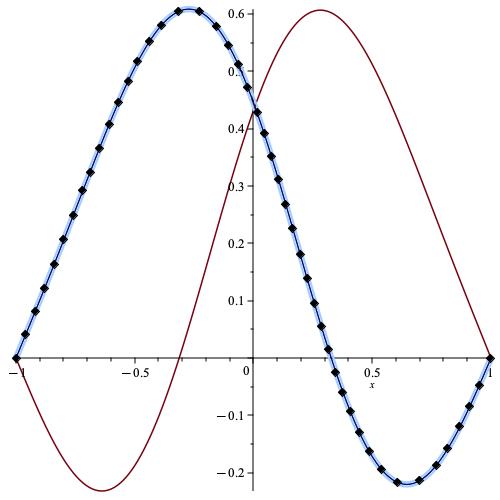}
\captionof{figure}{Here, $\varphi$ is the upper curve on the left.}
\label{fig:airy32}
\end{center}
\end{figure}
\end{example}

\begin{example}\label{exa3}
In this example we exhibit the oscillations in the more general case where the weight function vanishes on an interval, rather than just a point, as in \cite{Rich18}.

Here, the piecewise continuous weight function $r$ is given by

\[ r(x) = \left \{ \begin{array}{lll}
		1,	&	\mbox{if  \ \ $ -2 \leq x < -1$}, \\
		0 ,	&	\mbox{if  \ \ $-1 \leq x < 1 $}. \\
		-1 ,	&	\mbox{if  \ \ $1 \leq x \leq 2 $}. \\
		\end{array}
	\right.
\]

and $q(x)=q > 0$, is a constant function. In order to simplify the notation, the Dirichlet problem is defined by considering \eqref{e1} (with $q$ there replaced by $-q$ and $r$ by $-r$) subject to $y(-2)=0=y(2)$. Of course, this changes the sign of $G(x)$ in general, if we wish to maintain ${\rm Im\, }\lambda >0$.

In the following we use the following simplified notation: $\sin\sqrt{q-\lambda}= \sin(\sqrt{q-\lambda})$, $\sin2\sqrt{q} = \sin (2\sqrt{q})$, $\sin\sqrt{q-\lambda}x = \sin (\sqrt{q-\lambda}x)$ etc.

Thus, the solution of the problem satisfying $y(-2)=0$, $y^\prime(-2)=1$ is given by
$$y(x) = \frac{\sin 2\sqrt{q-\lambda}}{\sqrt{q-\lambda}}\left(\cot\ 2 \sqrt{q-\lambda} \sin \sqrt{q-\lambda}\, x+\cos\! \sqrt{q-\lambda}\, x\right),$$
for $-2 \leq x \leq -1$. 

On the other hand, on the interval $-2\leq x  < 1$, the solution $y$ is given by,
\begin{equation*}
y(x) = \frac{1}{\sqrt{q-\lambda}}\left ( \left \{A-B \right \} \sin \sqrt{q}x\right)  + C\cos\sqrt{q}x 
\end{equation*}
where 
\begin{eqnarray*}
A &= & \frac{(\sin\sqrt{q-\lambda}\, \sqrt{q} \cot \sqrt{q} +\sqrt{q-\lambda}\cos{\sqrt{q-\lambda})\cot\sqrt{q}}}{\sqrt{q}\csc\sqrt{q}},\\
B &=& \sin\sqrt{q-\lambda}\csc{\sqrt{q}},\\
C &=& \frac{1}{\sqrt{q}\sqrt{q-\lambda}\csc\sqrt{q}}\sin\sqrt{q-\lambda}\sqrt{q}\cot{\sqrt{q}}+ \sqrt{q-\lambda}\cos \sqrt{q-\lambda}.
\end{eqnarray*}
Finally, the expression for $y$ on $1 \leq x \leq 2$ is given by
\begin{eqnarray}
y(x) = \left ( \frac{1}{\sqrt{q+\lambda}\sqrt{q}\sqrt{q-\lambda}}\left \{ A^\prime + B^\prime\right \}\right) \sin \sqrt{q+\lambda}x + \nonumber \\
\left ( \frac{1}{\sqrt{q+\lambda}\sqrt{q}\sqrt{q-\lambda}}\left \{ C^\prime + D^\prime\right \}\right) \cos \sqrt{q+\lambda}x \label{disp}
\end{eqnarray}
where ...
\begin{eqnarray*}
A^\prime &= &\sqrt{q+\lambda}\sqrt{q} \sin\sqrt{q+\lambda}\cos2\sqrt{q}\sin \sqrt{q-\lambda} \\
&& + \sqrt{q-\lambda} \sqrt{q+\lambda}\sin\sqrt{q+\lambda}\sin 2\sqrt{q}\cos \sqrt{q-\lambda},\\
B^\prime &=& \sqrt{q}\sqrt{q-\lambda}\cos{\sqrt{q+\lambda}} \cos \sqrt{q-\lambda}\cos 2\sqrt{q}\\
&& - q\,\cos\sqrt{q+\lambda}\sin \sqrt{q-\lambda}\sin 2\sqrt{q},\\
C^\prime &=& \sqrt{q} {\sqrt{q+\lambda}} \cos\sqrt{q+\lambda}\cos 2\sqrt{q}\sin \sqrt{q-\lambda}\\
&& + \sqrt{q+\lambda}\sqrt{q- \lambda}\cos \sqrt{q+\lambda} \sin 2\sqrt{q} \cos \sqrt{q- \lambda} ,\\
D^\prime & = & q \sin 2\sqrt{q}\sin \sqrt{q+\lambda} \sin \sqrt{q- \lambda},\\
&& - \sqrt{q} \cos 2\sqrt{q} \sqrt{q-\lambda} \cos \sqrt{q-\lambda}\sin\sqrt{q+\lambda}.
\end{eqnarray*}

The dispersion relation for the eigenvalues of this non-definite problem is given by setting $y(2)=0$ in \eqref{disp} and solving for $\lambda \in \C$. Using the Maple$^\copyright$  \ subroutine {\it RootFinding} we find a non-real eigenvalue $\lambda \approx 6.29625i$ in the case where $q=8$. 
\begin{center}
\includegraphics[width=2in]{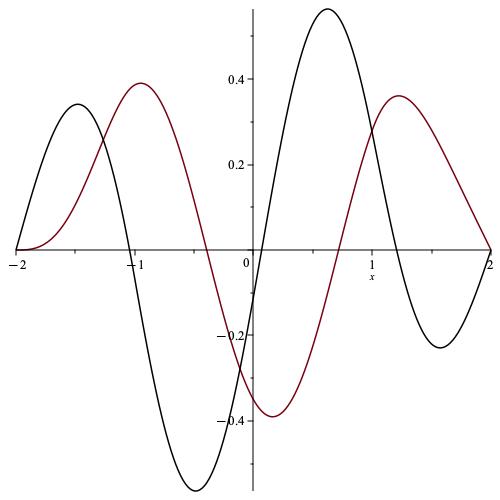}
\captionof{figure}{Here, the left-most curve is $\varphi$, the real part.}
\end{center}
It appears as if $G(x)<0$ in this case in which case Lemma~\ref{lem5} (1) applies.
\end{example}

\begin{example}\label{exa4}
Returning to Example~\ref{exa1} above, 
$$-y^{\prime\prime}(x) - q\,y(x)= \lambda \sgn(x)y(x), \quad\quad y(-1)=0,\ y(1)=0.$$
where now $q = \pi^2/4 + 2$, we get a non-real eigenvalue pair $\lambda \approx \pm 3.8741i$.
. 

Writing $y = \varphi  + \psi i$, the graphs of $\varphi, \psi,$ and $G(x)$ are reproduced below,
\begin{center}
\includegraphics[width=3in]{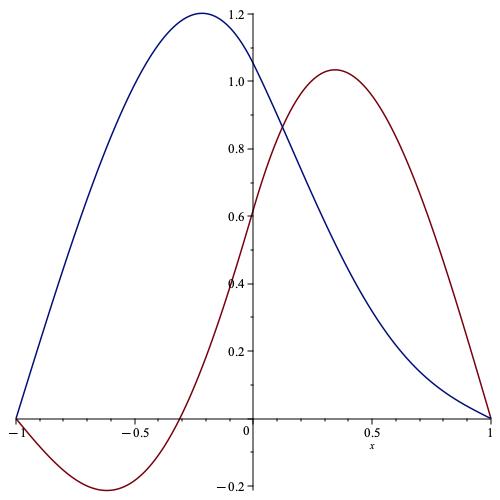}
\captionof{figure}{Here, $\varphi$ is the left-most curve of the graph on the left, $\psi$ being the other. Note the positivity of the real part in $(-1,1)$. Again, $G(x)<0$, here. 
\label{fig31p}}
\end{center}
\vskip0.25in

\end{example}

\section{Sturm theory for non-real eigenfunctions}

In this section the word "eigenvalue" refers to a non-real eigenvalue of \eqref{e1}-\eqref{e1p} and all questions here, unless otherwise specified, refer to the case where the weight function has a unique turning point around which $r(x)$ changes sign. For example, we can ask:

\begin{itemize}

\item[1.] Under what conditions does there necessarily exist an eigenfunction whose real or imaginary part does not vanish in $(a, b)$?

Now, \eqref{e1}-\eqref{e1p} always has a finite number of pairs of non-real eigenvalues by Lemma~\ref{lem4}, so this is a sensible question. In Example~\ref{exa2} there is a complex ghost whose real and imaginary part both have a zero in $(a, b)$.  In this case, a simple calculation shows that $M=1$ in Lemma~\ref{lem4}, and thus there can only be at most one pair of non-real eigenvalues.  Since this eigenvalue is one of the only possible pairs of eigenvalues of the problem, the answer is negative in this case.

On the other hand, in Example~\ref{exa4}, for $q = \pi^2/4 + 2$ we have an eigenvalue pair $\lambda \approx \pm 3.8741i$ with an eigenfunction having a real part that is seemingly positive in $(-1, 1)$. Once again, $M=1$ in Lemma~\ref{lem4} and thus this is the only non-real pair.

In relation to 1) above, we can ask for a weaker result, that is, 

\item [2.] Can a non-real eigenfunction $y(x, \lambda)$ vanish for some $x\in (a, b)$?

Generally speaking this cannot occur in the case of one turning point such as in Richardson \cite{Rich18} or, more generally, as in Theorem~\ref{th3}. This is because of the separation of the interior zeros of the real and imaginary parts of the eigenfunction.

However, the possibility of a non-real eigenfunction vanishing in the interior of an interval may be true in the case of two turning points as the results in \cite{Mervis2016} show. There, an example of a non-definite problem with two turning points was found in which a non-real eigenfunction vanished as a function of $x$ inside $(a, b)$. The result is therefore likely true in the case of more than two turning points, but this has not been verified yet. In \cite{abm82} it was shown that if the weight function has $n$ turning points, then any non-real eigenfunction can vanish at most $n-1$ times in $(a, b)$.

\item [3.] Under what conditions can we infer the existence of an eigenfunction, corresponding to an eigenvalue of largest modulus, that does not vanish in $(a, b)$?

In the case of one turning point or as in Theorem~\ref{th3} this is clear as eigenfunctions cannot vanish in the interior because of the separation property. In particular, see Example~\ref{exa4} where we have only one-non-real pair. This statement is then reminiscent of the Perron-Frobenius Theorem \cite{hj} in the finite dimensional (matrix) case. (See Corollary 3.2 in \cite{maps} for a related result.) Since our boundary value problem always has a finite number of pairs of non-real eigenvalues by Lemma~\ref{lem4}, this too is a ponderable question. To what extent is it true in cases where there are multiple turning points?
\end{itemize}

\end{document}